\documentclass[twocoloumn]{IEEEtran}

\makeatletter
\def\ps@headings{%
\def\@oddhead{\mbox{}\scriptsize\rightmark \hfil \thepage}%
\def\@evenhead{\scriptsize\thepage \hfil \leftmark\mbox{}}%
\def\@oddfoot{}%
\def\@evenfoot{}}
\makeatother 

\usepackage{amsmath}
\usepackage{txfonts}
\usepackage{amssymb}
\usepackage{booktabs}
\usepackage{graphicx}
\usepackage[tight,footnotesize]{subfigure}
\usepackage{cite}
\usepackage{url}
\usepackage{tabularx}
\usepackage{algorithm}
\usepackage{algorithmic}
\usepackage{multirow}
\usepackage{paralist}
\usepackage{xspace}
\usepackage{color}

\newtheorem{thm}{Theorem}

\newcommand{\minC}{\textsc{MinCost}\xspace}
\newcommand{\maxC}{\textsc{MaxCred}\xspace}
\newcommand{\minCDP}{\textsc{MinCost-DP}\xspace}
\newcommand{\maxCDP}{\textsc{MaxCred-DP}\xspace}
\newcommand{\minCMCF}{\textsc{MinCost-MCF}\xspace}
\newcommand{\maxCMCF}{\textsc{MaxCred-MCF}\xspace}
\newcommand{\maxCTwoF}{\textsc{MaxCred-2F}\xspace}
\newcommand{\minCAnn}{\textsc{MinCost-CC}\xspace}
\newcommand{\maxCAnn}{\textsc{MaxCred-CC}\xspace}
\newcommand{\minCS}{\textsc{MinCost-Stochastic}\xspace}
\newcommand{\maxCS}{\textsc{MaxCred-Stochastic}\xspace}

\newcommand{\expect}[1]{\mathbb{E}\left\{#1\right\}}
\newcommand{\defequiv}{\mbox{\raisebox{-.3ex}{$\overset{\vartriangle}{=}$}}}

\newcommand{\bv}[1]{{\boldsymbol{#1} }}
\newcommand{\script}[1]{{{\cal{#1} }}}

\newtheorem{theorem}{Theorem}[section]
\newtheorem{lemma}[theorem]{Lemma}

\newenvironment{proof}[1][Proof]{\begin{trivlist}
\item[\hskip \labelsep {\bfseries #1}]}{\end{trivlist}}

\definecolor{brown}{cmyk}{0,0.81,1,0.60}
\definecolor{magenta}{rgb}{0.4,0.7,0}
\definecolor{gray}{rgb}{0.5,0.5,0.5}
\definecolor{red}{rgb}{1,0,0}
\definecolor{green}{rgb}{0.5,0,0.5}
\definecolor{blue}{rgb}{0,0,1}

\newcommand{\qed}{\nobreak \ifvmode \relax \else
      \ifdim\lastskip<1.5em \hskip-\lastskip
      \hskip1.5em plus0em minus0.5em \fi \nobreak
      \vrule height0.75em width0.5em depth0.25em\fi}

\usepackage{paralist}

\begin{document}

\title{Optimizing Information Credibility in Social Swarming Applications}

\author
{ Bin Liu$^*$, Peter Terlecky$^\ddagger$, Amotz Bar-Noy$^\ddagger$,
Ramesh Govindan$^*$, Michael J. Neely$^*$
\\

\thanks{
Research was sponsored by the Army Research Laboratory and was
accomplished under Cooperative Agreement Number W911NF-09-2-0053.
The views and conclusions contained in this document are those of
the authors and should not be interpreted as representing the
official policies, either expressed or implied, of the Army Research
Laboratory or the U.S. Government.  The U.S. Government is
authorized to reproduce and distribute reprints for Government
purposes notwithstanding any copyright notation here on.}

\begin{tabular}{cc}
  $^*$University of Southern California, Los Angeles &
  $~^\ddagger$City University of New York
\end{tabular}

}

\maketitle

\begin{abstract}
With the advent of smartphone technology, it has
become possible to conceive of entirely new classes of applications.
\emph{Social swarming}, in which users armed with smartphones are
directed by a central \emph{director} to report on events in the physical
world, has several real-world applications: search and rescue,
coordinated fire-fighting, and the DARPA balloon hunt challenge.
In this paper, we focus on the following problem: how does the
director optimize the selection of reporters to deliver credible
\emph{corroborating} information about an event.
We first propose a model, based on common intuitions of believability,
about the credibility of information.
We then cast the problem posed above as a discrete optimization
problem, and introduce optimal centralized solutions and an
approximate solution amenable to decentralized implementation whose
performance is about 20\% off on average from the optimal (on
real-world datasets derived from Google News) while being 3 orders
of magnitude more computationally efficient.
More interesting, a time-averaged version of the problem is amenable
to a novel stochastic utility optimization formulation, and can be
solved optimally, while in some cases yielding decentralized
solutions.
To our knowledge, we are the first to propose and explore the problem
of extracting credible information from a network of smartphones.
\end{abstract}

\section{Introduction}
\label{sec:introduction}

With the advent of smartphone technology, it has become possible to
conceive of entirely new classes of applications.
Recent research has considered personal reflection~\cite{PEIR}, social
sensing~\cite{CenceMe}, lifestyle and activity
detection~\cite{Activity}, and advanced speech and image processing
applications~\cite{CyberF}.
These applications are enabled by the programmability of
smartphones, their considerable computing power, and the presence of
a variety of sensors on-board.

In this paper, we consider a complementary class of potential
applications, enabled by the same capabilities, that we call
\emph{social swarming}.
In this class of applications, a \emph{swarm} of users, each armed
with a smart phone, cooperatively and collaboratively engages in one
or more tasks.
These users often receive instructions from or send \emph{reports}
(a video clip, an audio report, a text message, or etc.) to a
\emph{swarm
  director}.
Because directors have a global view of information from different
users, directors are able to manage the swarm efficiently to achieve the
task's objectives.
Beyond the obvious military applications, there are several civilian
ones: search and rescue, coordinated fire-fighting, and the DARPA
balloon hunt challenge\footnote{
\texttt{http://www.crn.com/networking/222000334}}.

In these applications, an important challenge is to obtain
\emph{credible} (or believable) information.
In general, sociologists have observed three ways in which
believable information might be obtained~\cite{Dyads}:
\emph{homophily}, by which people believe like-minded people;
\emph{test-and-validate}, by which the recipient of information
tests the correctness of the information; and \emph{corroboration},
where the belief in information is reinforced by several sources
reporting the same (or similar) information.
The process by which humans believe information is exceedingly
complex, and an extended discussion is beyond the scope of this paper.

Instead, our focus is on simple and tractable models for corroboration
in social swarming type applications.
Specifically, the scenario we consider is the following.
Suppose that an \emph{event} (say, a balloon sighting) is reported
to a swarm director.
%
The director would like to corroborate this report by obtaining
reports from other swarm members: which reporters should she
select?
We call this the \emph{corroboration pull} problem.
Clearly, asking every swarm member to report is unnecessary, at best:
swarms can have several hundred participants, and a video report from
each of them can overwhelm the network.
Thus, intuitively, the director would like to \emph{selectively}
request reports from a subset of swarm members, while managing the
network resources utilized.

In this paper, we formalize this intuition and study the space of
corroboration pull formulations.
Our contributions are three-fold.
{\bf $1)$}
We introduce a model for the credibility of reports. This
  model quantifies common intuitions about the believability of
  information: for example, that video is more believable than text,
  and that a reporter closer to an event is more believable than one
  further away (Section~\ref{sec:model}).
{\bf $2)$}
We then cast the \emph{one-shot} corroboration pull problem as a
  discrete optimization problem and show that it reduces to a
  multiple-choice knapsack problem with weakly-polynomial optimal
  solutions. We develop strongly-polynomial, but inefficient,
  solutions for the case when the number of formats is fixed, and an
  optimal algorithm for the case of two formats. Finally, we derive an
  approximation algorithm for the general case that leverages the
  structure of our credibility model. This algorithm is about 20\% off
  the optimal, but its running time is 2-3 orders of magnitude faster
  than the optimal algorithm, a running time difference that can make the
  difference between winning and losing in, say, a balloon hunt.
{\bf $3)$}
We then show that, interestingly, the \emph{renewals} version of
  the problem, where the goal is to optimize corroboration pull in a
  time-averaged sense, can be solved optimally, while admitting a
  completely decentralized solution.

%


\section{Terminology, Model, and Optimization Formulation}
\label{sec:model}

As smart phones proliferate, social swarming applications are likely
to become increasingly common.
In this paper, we consider a constrained form of a social swarming
application in which $N$ participants, whom we call \emph{reporters},
collaboratively engage in a well-defined task.
Each reporter is equipped with a smart phone and directly reports to
a swarm director
using the 3G/EDGE network.
A reporter may either be a human being or a sensor (static, such as a
fixed camera, or mobile, as a robot).
A \emph{director} (either a human being, or analytic software)
assimilates these reports, and may perform some
actions based on the content of these combined reports.

Our setting is simplified in many ways.
For now, we consider a situation where reporters cooperate, and are
therefore benign: we leave a consideration of malicious reporting to
future work.
Similarly, we have implicitly assumed an always available 3G/EDGE
network, and have not considered network dynamics (such as the
availability of opportunistic WiFi networks).
We believe this assumption can be relaxed using techniques from our
prior work~\cite{Ra10}, 
but have left an exploration of this to future work.
Despite these simplifications, we show that the problem space has
sufficient richness in and of itself.

Each reporter reports on an \emph{event}.
The nature of the event depends upon the social swarming application:
for example, in a search and rescue operation, an event corresponds to
the sighting of an individual who needs to be rescued; in the balloon
hunt, an event is the sighting of a balloon.
Events occur at a particular \emph{location}, and multiple events may
occur concurrently either at the same location or at different
locations.

Reporters can transmit reports of an event using one of several
formats: such as a video clip, an audio clip, or a text message
describing what the report sees.
Each report is a form of \emph{evidence} for the existence of the
event.
As we discuss below, different forms of evidence are ``believed'' to
different extents.
In general, we assume that each reporter is capable of generating
$R$ different report formats, denoted by $f_j$, for $1 \le j \le R$.
However, different formats have different costs to the network: for
example, video or audio could consume significantly higher
transmission resources than, say, text.
We denote by $e_j$ the cost of a report $f_j$: for ease of
exposition, we assume that reports are a fixed size so that all
reports of a certain format have the same cost (our results can be
easily generalized to the case where report costs are proportional
to their length).

Finally, reporters can be mobile, but we assume that the director is
aware of the location of each reporter.
In our  problem formulation, we ignore the cost of sending periodic
location updates to the director.
In practice, this may be a reasonable assumption for three reasons.
First, the cost of location updates may be amortized over other
context aware applications that may be executing on
the smart phone.
Second, although this cost may be significant, it adds a fixed cost to
our formulations and does not affect the results we present in the
paper.
Finally, the absolute cost of the location updates themselves is
significantly less than the cost of video transmissions, for example.

Now, suppose that the director in a swarming application has heard,
through out of band channels or from a single reporter, of the
existence of an event $E$ at location $L$.
To verify this report, the director would like to request
\emph{corroborating} reports from other reporters in the vicinity of
$L$.
Which reporters should she get corroborating reports from?
What formats should those reporters use?

To understand this, recall that the goal of corroboration is to
increase the director's belief in the occurrence of the event.
How much should the director believe a specific reporter?
Or, equivalently, what is the \emph{credibility} of a report?

In general, this is a complex sociological and psychological question
which, at the moment, is not objectively quantifiable.
However, in this paper,  we model the credibility of the report using
two common intuitions about credibility.
The first intuition is based on the maxim ``seeing is believing'': a
video report is more credible than a text report.
We extend this maxim in our model to incorporate other formats, like
audio: audio is generally less credible than video (because, while
it gives some context about an event, video contains more context),
but more credible than text (for a similar reason).
Of course, this is an assumption: video and audio can be just as
easily doctored as text.
Recall that our model, for now, assumes cooperative non-malicious
elements: in future work, we plan to discuss how to model the
credibility of reports
in the presence of malicious
elements.
Moreover, as we shall show later, many of our results are insensitive
to the exact choice of the credibility model.

Our second intuition is based on the often heard statement ``I'll
believe someone who was there'', suggesting that proximity of the
reporter to an event increases the credibility of the report.
More precisely, a report $A$ generated by a reporter at distance $d_a$
from an event has a higher credibility than a report $B$ generated by
a reporter at a distance $d_b$, if $d_a < d_b$.
This is also a simplified model: the real world is more complex, since
the complexity of the terrain, or line of sight, may matter more than
geometric distance.

While are many different ways in which we can objectively quantify the
credibility of a report given these intuitions, we picked the
following formulation.
Let $S_i$ be the position of reporter $i$, $L$ be the position of
event $E$ and $c_{i,j}(S_i, L)$ be the credibility of the report
generated by reporter $i$ when report format $f_j$ is used.
We define $c_{i,j}(S_i, L)$ as:
\begin{eqnarray} \label{eq:text}
{c_{i,j}}({S_i},L) = \left\{ {\begin{array}{*{20}{c}}
   {{{\gamma _j}}/{d{{({S_i},L)}^{{\delta _j}}}},} & {\text{ if }{h_0} < d({S_i},L)}  \\
   {{{\gamma _j}}/{h_0^{{\delta _j}}},} & {\text{ if }d({S_i},L) \le {h_0}}  \\
\end{array}} \right.
\end{eqnarray}
with  $1\le j \le R$, ${\gamma_1} \le {\gamma_2} \le \cdots \le
{\gamma_R}$, and ${\delta _1} > {\delta _2} > \cdots > {\delta _R}$.
Here, $d(.)$ is the Euclidean distance between points, $h_0$ is a
certain minimum distance to avoid division by zero as well as to
bound the maximum credibility to a certain level, $\gamma_j$ is a
constant of proportionality implying the maximum achievable
credibility of report format $f_j$, and the credibility decays
according to a power-law with exponent $\delta_j$ when format $f_j$
is used.

Our credibility model incorporates the two intuitions described above
as follows.
The intuition about the credibility being dependent on proximity is
captured by the power-law decay with distance.
The intuition about the higher credibility of the video compared to
text is captured by having a larger $\gamma$ and
a smaller exponent for video.

This model can be extended to incorporate \emph{noise} or
\emph{confusion}.
For example, poor visibility or audible noise near a reporter may,
depending upon the format used, reduce the believability of a
report.
The intensity of point sources of noise can be modeled as a function
that decays with distance:
\begin{equation} \label{eq:noise}
G_1(S_i,O_1) = \frac{1}{{[1 + d{{(S_i,O_1)}]^{1/\sigma_1} }}}
\end{equation} where $S_i$ is the position of reporter $i$, $O_1$ is the
position of noise source 1, and $\sigma_1$ represents the strength
and effective range of noise source 1.
Then, if for reporter $i$ and event $E$, the original credibility
without noise is $c_{i,j}(S_i, L)$, then the credibility with $X$
noise sources should be
\begin{equation} \label{eq:cnoise}
c_{i,j}'(S_i, L) = c_{i,j}(S_i, L)\prod\nolimits_{p = 1}^X {\left(
{1 - G_p(S_i,O_p)} \right)}
\end{equation}
Noise sources effectively increase the distance of the reporter from
the event, reducing his or her credibility.
As we show later, our solutions can incorporate this form of noise
without any modification.

Although we have assigned objective quantitative values to
credibility, belief or disbelief is often qualitative and subjective.
Thus, we don't expect swarm directors to make decisions based on the
exact values of credibility of different reports, but rather to
operate in one of two modes: $a)$ ask the network to deliver
corroborating reports whose total credibility is above a certain
threshold, while minimizing cost, or $b)$ obtain as much corroborating
information that they can get from the network for a given cost.
We study these two formulations, respectively called \minC and
\maxC.

Before doing so, there are two questions to be answered: What is the
value of the credibility of a collection of corroborating reports?
What is the physical/intuitive meaning of a threshold on the
credibility?
For the first question, there are many possible answers and we
consider two.
With an \emph{additive corroboration function}, the  total credibility
is simply the sum of the individual credibilities.
More generally, with a \emph{monotonically-increasing corroboration
  function}, the total credibility increases monotonically as a
function of the sum of the individual credibilities.
The second question is important because it can help directors set
thresholds appropriately.
The intuition for a particular threshold value $C$ can be explained as
follows.
Suppose a director would be subjectively satisfied with 3
corroborating video clips from someone within 10m of an event.
One could translate this subjective specification into a threshold
value by simply taking the sum of the credibilities of 3 video reports
from a distance of 10m.

In the next two sections, we formally define \minC and \maxC, and then
consider two problem variants: a \emph{one-shot} problem which seeks
to optimize reporting for individual events, and a \emph{renewals}
problem which optimizes reporting over a sequence of event arrivals.

\section{The One-Shot Problem}
\label{sec:num}

In this section, we formally state the \minC and \maxC formulations
for the additive corroboration function and in the absence of noise,
discuss their complexity, develop optimal solutions for them, and then
explore an approximation algorithm that leverages the structure of the
credibility function for efficiency.
We conclude with a discussion of extensions to the formulations for
incorporating the impact of noise sources, and for
monotonically-increasing corroboration function.
Our exposition follows the notation developed in the previous section,
and summarized in Table~\ref{tab:notation}.

\begin{table}[t]
  \centering
  \scriptsize
  \caption{Notation}
  \setlength\doublerulesep{0pt}
  \setlength{\baselineskip}{1.5\baselineskip}
\begin{tabular}{|c||c|}
\hline
$N$ & the total number of available reporters\\
\hline
$c_{i,j}$ & the short form of (\ref{eq:text}) in a given event \\
\hline
$R$ & the total number of report formats  \\
\hline
$e_j$ & the cost when using report format $f_j$  \\
\hline
$C$ & the target credibility in \minC  \\
\hline
$A$ & the dynamic programming process of \minC \\
\hline
$B$ & the cost budget in \maxC  \\
\hline
$D$ & the dynamic programming process of \maxC \\
\hline
\end{tabular} \label{tab:notation}
\end{table}

\subsection{\minC and \maxC: Problem Formulation and Complexity}
\label{sec:complexity}


\subsubsection{Problem Formulations}

Recall that, in Section~\ref{sec:model}, we informally defined the
\minC problem to be: what is the minimum cost that guarantees total
credibility $C>0$?
\minC can be stated formally as an optimization problem:
\begin{eqnarray}
{\rm{Minimize:}}&&\sum_{i=1}^N\sum_{j=1}^R x_{i,j}e_{j} \label{eq:op10}\\
{\text{Subject to:}}&&\sum_{i=1}^N\sum_{j=1}^R x_{i,j}{c_{i,j}}\ge C\nonumber\\
{\rm{  }}&&{x_{i,j}} \in \{ 0,1\},\forall i\in \{1,...,N\}, \forall j\in \{1,...,R\}\nonumber\\
{\rm{  }}&&\sum_{j=1}^R x_{i,j} \leq 1, \forall i\in
\{1,...,N\}\nonumber
\end{eqnarray}
where $x_{i,j}$ is a binary variable that is $1$ if reporter $i$
uses format $f_j$, and $0$ otherwise.

Analogously, we can formulate \maxC (the maximum credibility that
can be achieved for a cost budget of $B>0$) as the following
optimization problem:
\begin{eqnarray}
{\rm{Maximize:}}&&\sum_{i=1}^N\sum_{j=1}^R x_{i,j}c_{i,j} \label{eq:op30}\\
{\text{Subject to:}}&&\sum_{i=1}^N\sum_{j=1}^R x_{i,j}{e_j}\le B\nonumber\\
{\rm{  }}&&{x_{i,j}} \in \{ 0,1\},\forall i\in \{1,...,N\}, \forall j\in \{1,...,R\}\nonumber\\
{\rm{  }}&&\sum_{j=1}^R x_{i,j} \leq 1, \forall i\in
\{1,...,N\}\nonumber
\end{eqnarray}

\subsubsection{On the Complexity of \minC and \maxC}
If, in the above formulation, the cost $e_j$ is also dependent on the
identity of the reporter (and therefore denoted by $e_{i,j}$), the
\maxC problem can be shown to be a special instance of the
Multiple-Choice Knapsack Problem (MCKP, \cite{knapsack}).
Moreover, the special case of one format (and $e_{i,j}=e_i$) is the
well-known Knapsack problem (KP) which is NP-hard.
However, when the cost is dependent only on the format
(i.e., $e_{i,j}=e_j$), we can state the following theorem, whose proof
(omitted for brevity)
uses a reduction from the original Knapsack problem.

\begin{theorem}
\minC and \maxC are NP-Hard. \label{the1}
\end{theorem}
%

\subsection{Optimal Solutions}
\label{sec:opt}

Despite Theorem~\ref{the1}, it is instructive to consider optimal
solutions for the two problems for two reasons.
First, for many social swarming problem instances, the problem sizes
may be small enough that optimal solutions might apply.
Second, optimal solutions can be used to calibrate an approximation
algorithm that we discuss later.
In this section, we discuss two classes of optimal solutions for \minC
and \maxC, with different tradeoffs: one based on \emph{dynamic programming},
and another based on a \emph{min-cost flow formulation}.

\subsubsection{Dynamic Programming}

Since there exist optimal, weakly-polynomial algorithms for MCKP, it
is natural that similar algorithms exist for \minC and \maxC.
We describe these algorithms for completeness, since we use them in a
later evaluation.

For \minC (\ref{eq:op10}), we can write $y_{i,j} = 1 - x_{i,j}$,
where $y_{i,j} \in \{0, 1\}$, and then we have:
\begin{eqnarray}
 &&\sum\limits_{i = 1}^N {\sum\limits_{j = 1}^R {{e_j}} }  - {\rm{Maximize  }}\sum\limits_{i = 1}^N {\sum\limits_{j = 1}^R {{y_{i,j}}{e_j}} }  \label{eq:op20}\\
 {\text{Subject to:}}&&\sum\limits_{i = 1}^N {\sum\limits_{j = 1}^R {{y_{i,j}}{c_{i,j}}} }  \le \sum\limits_{i = 1}^N {\sum\limits_{j = 1}^R {{c_{i,j}}} }  - C = W \nonumber\\
 &&{y_{i,j}} \in \{ 0,1\} ,\forall i \in \left\{ {1, \ldots ,N} \right\}{\rm{,}}\forall j \in \left\{ {1, \ldots ,R} \right\} \nonumber\\
 &&\sum\limits_{j = 1}^R {{y_{i,j}}}  \ge R - 1,\forall i \in \left\{ {1, \ldots ,N} \right\} \nonumber
\end{eqnarray} where the minimization problem (4) has been transformed
into a maximization problem, and the notation in (6) emphasizes that
the first term in the total cost $\sum_{i=1}^N\sum_{j=1}^R e_j$ does
not depend on the $y_{i,j}$ variables to be optimized.
For a given event, the sum of the $c_{i,j}$ values is a constant, and
so $W$ is also a constant.

This optimization problem can be solved by a dynamic programming
approach if we assume all $c_{i,j}$s are truncated to a certain
decimal precision, so that $c_{i,j} \in \{0, \zeta, 2\zeta,
\ldots\}$ where $\zeta$ is a discretization unit.
Then for any binary $y_{i,j}$ values that meet the constraints of
the above problem, the sum $\sum_{i=1}^N\sum_{j=1}^R y_{i,j}c_{i,j}$
takes values in a set $\script{W} \defequiv \{0, \zeta, 2\zeta,
\ldots, W\}$.
Note that the cardinality $|\script{W}|$ depends on $N$, $R$, the
$c_{i,j}$ values, and the discretization unit $\zeta$.
Now define $A(l, s)$ as the sub-problem of selecting reporters in the
set $\{1, \ldots, l\}$ subject to a constraint $s$.
Assuming $A(l,s)$ values are known for a particular $l$, we
recursively compute $A(l+1,s)$ for all $s \in \script{W}$ by:
\begin{eqnarray}
A(l+1, s) = \max[\phi^{(0)}(l, s), \phi^{(1)}(l,s), \ldots,
\phi^{(R)}(l,s)] \label{eq:dy10}
\end{eqnarray}
where $\phi^{(k)}(l,s)$ is defined for $k \in \{0, 1, \ldots, R\}$:
\begin{eqnarray*}
& \phi^{(k)}(l,s) \defequiv A(l, s - \sum_{j=1, j\neq k}^R c_{l,j})
+ \sum_{j=1,j\neq k}^R e_j &
\end{eqnarray*} This can be understood as follows: The value
$\phi^{(k)}(l,s)$ is the cost associated with reporter $l+1$ using
option $k \in \{0, 1, \ldots, R\}$ and then allocating reporters $\{1,
\ldots, l\}$ according to the optimal solution $A(l, s - \sum_{j=1, j
\neq k}^{R} c_{l,j})$ that corresponds to a smaller budget.
Note that option $k \in \{1, \ldots, R\}$ corresponds to reporter
$l+1$ using a particular format (so that $y_{l+1,k}=0$ for option $k$
and $y_{l+1,m} = 1$ for all $m \neq k$), and option $k=0$ corresponds
to reporter $l+1$ remaining idle (so that $y_{l+1,m} = 1$ for all
$m$).
The time complexity of this dynamic programming algorithm, called
\minCDP, is O($NR|\script{W}|$).
%
%
%
%

Similarly, \maxC can be solved using dynamic programming, yielding
an algorithm we label \maxCDP:
\begin{eqnarray}
&&D(l + 1,s) = \max [ {D(l,s),\mu^{(1)}(l,s),\mu^{(2)}(l,s),
\ldots ,\mu^{(R)}(l,s)}], \nonumber\\
&&\mu^{(k)}(l,s) \defequiv D(l,s - {e_k}) + {c_{l,k}} \text{ for } k
\in \{1, \ldots, R\} \label{eq:dy20}
\end{eqnarray}
with complexity $O(NR|\script{B}|)$. $\script{B} \defequiv \{0,
\eta, 2\eta, \ldots, B\}$, where $\eta$ is a discretization unit.



\subsubsection{Min-Cost Flow}
For a fixed number of formats, it is possible to define
strongly-polynomial, but not necessarily efficient, optimal algorithms
for \minC and \maxC.
These solutions are based on minimum-cost flow algorithms~\cite{clr}.
Define $\alpha_j$ to be the number of reporters reporting with
format $f_j$.
Define a report vector to be $(\alpha_1,\alpha_2,...,\alpha_R)$ and
an $(\alpha_1,\alpha_2,...,\alpha_R)$-assignment to be an assignment
of formats to reporters with $\alpha_j$ reporters reporting with
format $f_j$ for each $j\in\{1,..,R\}$.
We shall find an $(\alpha_1,\alpha_2,...,\alpha_R)$-assignment of
formats to reporters of maximum credibility.

We shall do so by transforming this problem to a min-cost flow
problem and applying a min-cost flow algorithm to obtain the
assignment of maximum credibility, in the following manner.
 %
 Assign nodes for each of the $N$ reporters and each of the $R$
 formats.
 Form a complete bipartite graph between the reporters and formats.
 Assign the edge connecting reporter $i$ to format $f_j$ with max
 capacity 1, min capacity 0, and cost $\max_{k\in
 \{1,...,N\}}\max_{l\in\{1,...,R\}} c_{k,l}-c_{i,j}$.
 Note that minimizing the set of such costs maximizes the set of
 $\{c_{i,j}\}$.
 Also, create a source node and a sink node.
 Connect the source node to each of the reporter nodes and give each
 edge max capacity 1, min capacity 0, and cost 0.
 Connect the sink node to each of the format nodes and give the edge
 connecting to format $f_j$ max capacity $\alpha_j$, min capacity
 $\alpha_j$, and cost 0.
 We shall call this network the \emph{credibility network}.
 This network has $O(N+R)$ vertices and $O(NR)$ edges.

This construction ensures that applying a min-cost flow algorithm to
the credibility network gives a minimum cost and maximum credibility
$(\alpha_1,...,\alpha_R)$-assignment of formats to reporters.
Using this construction, it is fairly easy to define an optimal
algorithm for \maxC, that we call \maxCMCF.
In this algorithm, we simply enumerate all possible
$(\alpha_1,...,\alpha_R)$-assignments, and find the highest total
credibility assignment that satisfies the cost budget $B$.
In a similar way, one can define \minCMCF.









There are $O(N^{R-1})$ possible report vectors.
The Enhanced Capacity Scaling algorithm~\cite{clr} solves the minimum
cost flow problem in time $O(|E|\log |V|(|E|+|V|\log |V|))$.
Thus, the Enhanced Capacity Scaling algorithm runs in time $O(NR\log
(N+R)(NR+((N+R)\log(N+R))))$ on the credibility network.
This leads to the following lemma:
\begin{lemma}
Both \minCMCF and \maxCMCF run in time $O(N^RR\log
(N+R)(NR+((N+R)\log(N+R))))$.
\end{lemma}

Note that when the number of formats $R$ is fixed, these algorithms
are polynomial in $N$.
In addition, when $|\script{B}|,|\script{W}|=\omega(N^{R-1}\log
(N+R)(NR+((N+R)\log(N+R))))$ these algorithms have lower asymptotic
complexity than their dynamic programming equivalents.

\subsection{Leveraging the Structure of the Credibility Function}
\label{sec:decentralized}

The solutions discussed so far do not leverage any \emph{structure}
in the problem.
Given an event and reporter locations, the credibility associated with
each report format is computed as a number and acts as an input to the
algorithms discussed.
However, there are two interesting structural properties in the
problem formulation.
First, for a given reporter at a given location, the credibility is
higher for a format whose cost is also higher.
Second, for reporters at different distances, the credibility decays
as a function of distance.
In this section, we ask the question: can we leverage this structure
to devise efficient approximation algorithms, or optimal
special-case solutions either for \maxC or \minC?

\subsubsection{An Efficient Optimal Greedy \maxC Algorithm for Two Formats}

When a social swarming application only uses two report formats (say,
text and video), it is possible to devise an optimal greedy \maxC
algorithm.
Assume each of the $N$ reporters can report with one of two formats,
$f_1$ or $f_2$, that reporters are indexed so that reporter $i$ is
closer to the event than reporter $k$, for $i<k$, and that credibility
decays with distance.
Furthermore, we assume that $e_{1}=\beta>1$ and $e_{2}=1$ and that
$c_{i,1}\ge c_{i,2}$ $\forall i\in \{1,...,N\}$: that the more
expensive format yields a higher credibility.

With these assumptions, the following algorithm, denoted \maxCTwoF,
finds an assignment with maximum credibility that falls within a
budget $B$ and runs in time $O(N^2)$.

\begin{algorithm}
\caption{Algorithm \maxCTwoF} \label{Alg1} \small $\textbf{INPUT}$:
$(c_{i,j})$: $i\in\{1,..,N\}, j\in\{1,2\}$; $(1,\beta)$; Budget $B$

Define $d_m \defequiv c_{m,1} - c_{m,2}$ for each $m \in \{1, \ldots, N\}$. \\
For $i \in \{0, \ldots, \min[\lfloor B/\beta \rfloor, N]\}$, do:
\begin{enumerate}
 \item Define $Y \defequiv \min[N-i, \lfloor B - \beta i \rfloor]$.
 \item Define the \emph{active set} $\script{A} \defequiv \{1, \dots,
   i + Y\}$, being the set of $i+Y$ reporters closest to the event.
 \item Define $\script{D}^*$ as the set of $i$ reporters in
   $\script{A}$ with the largest $d_m$ values (breaking ties
   arbitrarily).  Then choose format $f_1$ for all reporters $m \in
   \script{D}^*$, choose $f_2$ for $m \in \script{A} - \script{D}^*$,
   and choose ``idle'' for all $m \notin \script{A}$.
 \item Define $C_{MAX}^i$ as the total credibility of this assignment:
   \begin{eqnarray*}
    &C_{MAX}^i \defequiv \sum_{m\in\script{A}} c_{m,2} + \sum_{m\in\script{D}^*} d_m&
   \end{eqnarray*}
\end{enumerate}
$\textbf{OUTPUT}$:  $i^* \defequiv \arg\max_{i} C_{MAX}^i$.
\end{algorithm}

The output of this algorithm is the maximum credibility assignment of
formats to reporters.
We can prove that this algorithm is optimal.

\begin{thm}
  The above algorithm finds the solution $C_{MAX}$ to \maxCTwoF
  problem with budget $B$.
\end{thm}
\begin{proof} For each $i$, we first seek to find $C_{MAX}^i$, the
maximum credibility subject to having exactly $i$ reporters use the
expensive format $f_1$.
Using a simple interchange argument together with the fact that
credibility of each format is non-negative and non-increasing in
distance, we can show that there exists an optimal solution that
activates the set $\script{A}$ that consists of $i+Y$ reporters
closest to the event.
Indeed, if an optimal solution does not use the set $\script{A}$, we
can swap an idle reporter closer to the event with an active reporter
further from the event, without affecting cost or decreasing
credibility.

For each subset $\script{D} \subseteq\script{A}$ that contains $i$
reporters, define $C(\script{D})$ as the credibility of the
assignment that assigns reporters $m \in \script{D}$ the format
$f_1$, assigns the remaining reporters in $\script{A}$ the format
$f_2$, and keeps all reporters $m \notin \script{A}$ idle:
\begin{eqnarray*}
&C(\script{D}) = \sum_{m\in \script{A}} c_{m,2} + \sum_{m \in
\script{D}} d_m&
\end{eqnarray*}
Then $C(\script{D})$ is maximized by the subset $\script{D}^*$
containing the $i$ reporters in $\script{A}$ with the largest $d_m$
values.
This defines $C_{MAX}^i$, and $C_{MAX}$ is found by maximizing over
all possible $i$.
\end{proof}

We can analogously define a \minC version for two formats, but omit
it for brevity.
Currently, we have not been able to extend this type of algorithm to
3 formats and beyond, so this remains an interesting open problem.

\subsubsection{A Computationally-Efficient Approximation Algorithm}
\label{sec:ring_based}

The structure of our credibility function can also be used to reduce
computational complexity.
To understand this, recall that the dynamic programming algorithms
described above jointly optimized both reporter selection and format
selection.
In this section, we describe an approximation algorithm for \minC,
called \minCAnn, where the structure of the credibility function is
used to determine, for each reporter, the format that the reporter
should use.
As we shall show, \minCAnn has significantly lower run-times at the
expense of slight non-optimality in its results.

\minCAnn is based on the following intuition.
Close to the location of the event, even low-cost formats have
reasonable credibility.
However, beyond a certain distance, the credibility of low-cost
formats like text degrades significantly, to the point where even the
small cost of that format may not be justified.
Put another way, it is beneficial for a reporter to use that format
whose \emph{c}redibility per unit \emph{c}ost (hence \minCAnn) is
highest --- this gives the most ``bang for the buck''.
Thus, for a given reporter, its current distance $d$ from the location
of the event may pre-determine the format it uses.
Of course, this pre-determination can result in a non-optimal choice,
which is why \minCAnn is an approximation algorithm.

Formally, in \minCAnn, if, for a reporter $i$:
\begin{equation}
k^* = {{\arg\max}_k}\left[ {\frac{{{c_{i,k}}({S_i},L)}}{{{e_k}}}}
\right] \nonumber
\end{equation}
then reporter $i$ chooses format $f_{k^*}$.
This choice can be pre-computed (since it depends only upon the
credibility and cost models), but each reporter needs to recalculate
its choice of the report format whenever its relative distance to
the concerning event changes. The event locations that determine the
format $f_{k^*}$ chosen by a particular reporter $i$ form annular
regions about the reporter.

Once each reporter has made the format choice, it remains for the
director to decide which reporter(s) to select.
For \minCAnn, the minimum cost formulation is identical to
(\ref{eq:op20}), and with comparable complexity, but with two
crucial differences: both the constant $|\script{W}|$ and the
runtime now relate only to the number $N$ of reporters, not to $N
\times R$.
As we shall show below, this makes a significant practical difference
in runtime, even for moderate-sized inputs.
In \minCAnn, the dynamic programming process of (\ref{eq:dy10}) is
replaced by
\begin{equation}
{A(l + 1,s)} = \max \left\{ {{A(l,s)},{c_l} + {A(l,s - {e_l})}}
\right\} \label{eq:dy}
\end{equation}
where $c_l$ replaces $c_{l,j}$ in (\ref{eq:dy10}), since each reporter
precomputes its format of choice.
Compared with (\ref{eq:dy10}), the time complexity of (\ref{eq:dy})
is reduced to $O(N|\script{W}|)$ with a much smaller $|\script{W}|$
in general.
Notice that this time complexity is independent of $R$, the number
of report formats, greatly improving its computational efficiency at
the expense of some optimality.



Using steps similar to that presented in Section~\ref{sec:opt}, it is
possible to define a \maxCAnn approximation algorithm for maximizing
credibility.
We omit the details for brevity, but indicate that \minCAnn and
\maxCAnn still have weakly-polynomial asymptotic complexity, but are
computationally much more efficient than \minCDP and \maxCDP.




\vspace*{1ex}\noindent\textbf{Evaluation of \minCAnn.}
The approximation algorithms discussed above trade-off
optimality for reduced computational complexity.
As such, it is important to quantify this trade-off for practical
swarm configurations.
In this section, we compare \minCAnn with \minCDP\footnote{In some of
  our evaluations, we use $R=4$. In this regime, \minCDP is more
  efficient than \minCMCF, hence the choice.}: lack of space
prevents us from a more extensive evaluation, but we expect our
conclusions to hold in general.

Lacking data from social swarming applications, we use two different
data sets.
First, we carefully\footnote{We re-scaled reporter distances and did
several data cleaning operations:
  removing blog posts, handling duplicate reports etc. We omit a discussion of
  these for brevity.} manually mine Google News for interesting
events.
Searching for a specific set of keywords describing an event in Google
News retrieves a list of news items related to that event within 24
hours of occurrence of that event.
The event location is explicitly specified in the news items.
Each news item has a location, which is assumed to be the location of
a reporter.
We use the event location and report location as inputs to \minCAnn
and \minCDP.
In this paper, we present results from three events: an event of
\emph{regional} scope, a basketball playoff game between the Lakers
and the Jazz (31 reporters); an event of \emph{national} scope, the
passage of the healthcare reform bill (63 reporters); and an event of
\emph{global} scope, the opening of the Shanghai exposition (88
reporters).
Of course, this choice of a surrogate for social swarming is far from
perfect.
However, this data set gives a varied, realistic reporter location
distribution; since our algorithms depend heavily on location, we can
draw some reasonable conclusions about their relative performance.
That said, we also use a dataset generated from a random distribution
of reporters to ensure that we are not misled by the Google News data
set, but also to explore the impact of larger swarm sizes.

We are interested in two metrics: the \emph{optimality gap}, which
is the ratio of the min-cost obtained by \minCAnn to that obtained
by \minCDP; and the \emph{runtime} of the computation for each of
these algorithms.

\begin{figure}[!t]
\centering {
    \subfigure[Optimality gap]{\includegraphics[width=1.72in]{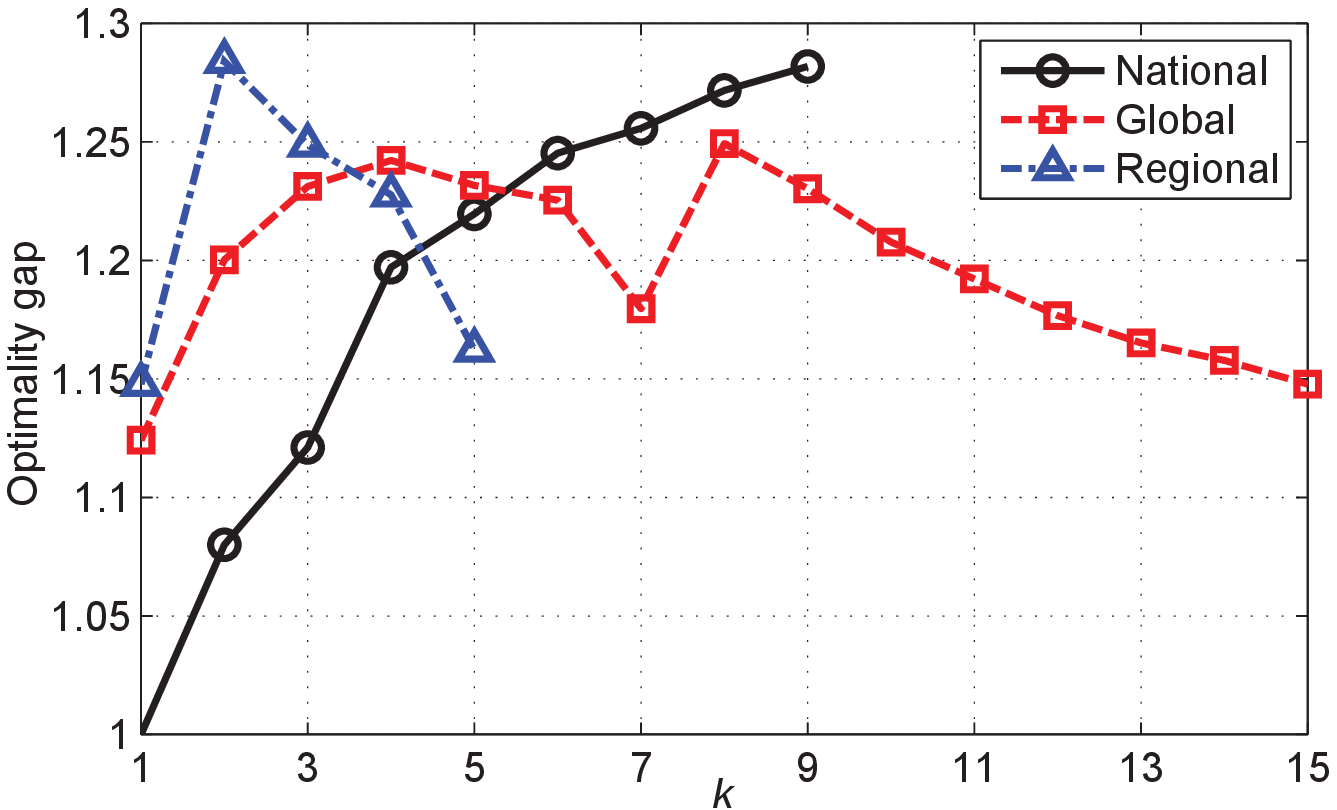}\label{fig:new5a}}
    \hfil
    \subfigure[Runtime (in microseconds)]{\includegraphics[width=1.67in]{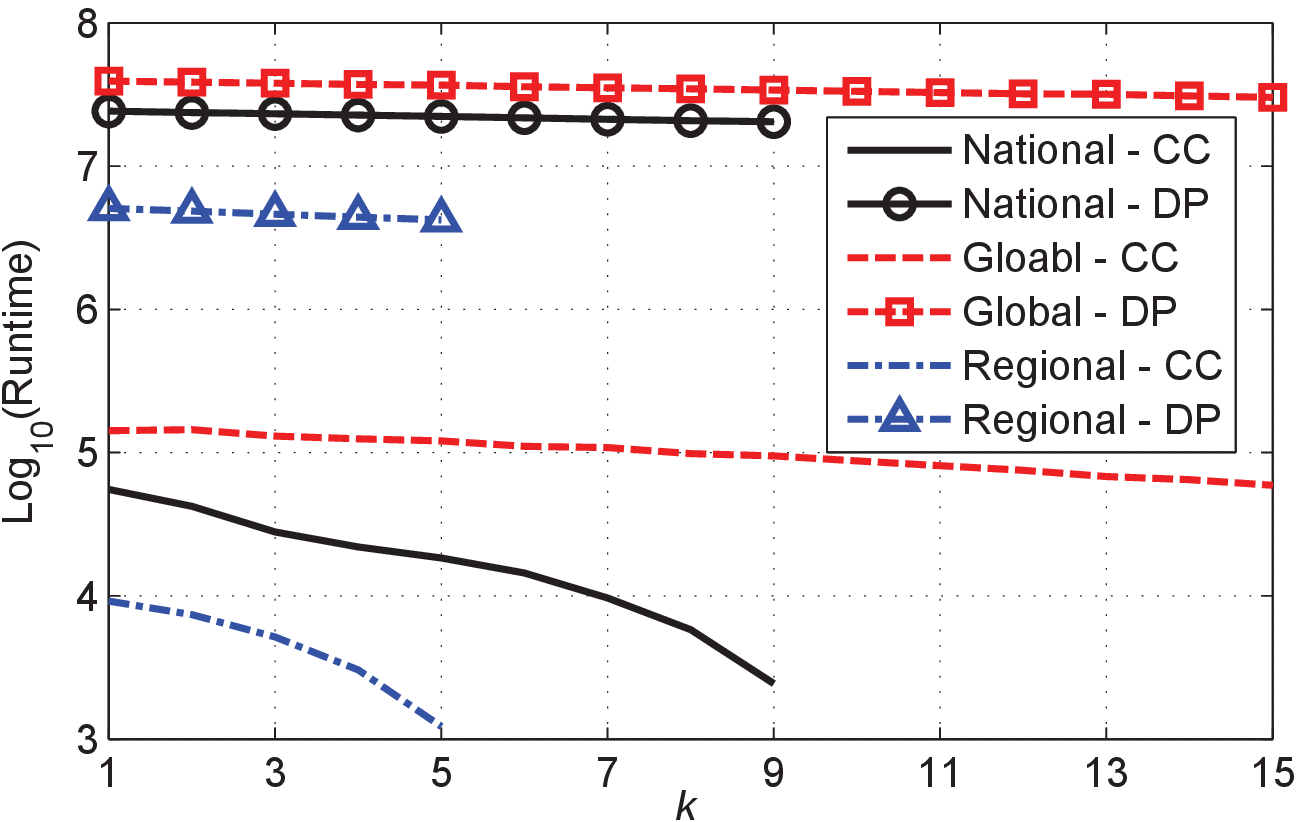}\label{fig:new5b}}
} \caption{Minimal cost of 4 formats with increasing $k$}
\label{fig:new5}
\end{figure}

Figure~\ref{fig:new5} plots these two metrics as a function of the
credibility threshold, expressed as a number $k$.
A value $k$ represents a credibility threshold corresponding to the
total credibility of $k$ reports of the highest cost format from a
distance $h_0$ (e.g., if $k$ is 3 and the highest cost format is
video, then the director is interested in obtaining credibility
equivalent to that from three video reports).
In this graph, we use four data formats with $h_0$, $\gamma_{1-4}$,
$\delta_{1-4}$ and the corresponding $e_{1-4}$ setting to 1, (1, 1,
1, 1), (2, 1.5, 1, 0.5) and (1, 2.2, 5.4, 13.7) respectively.
\emph{We have evaluated different numbers of data formats and
  different parameter settings and have obtained qualitatively similar
  results,} but omit a discussion of these for lack of space.

From Figure~\ref{fig:new5a}, the optimality gap is, on average
19.7\%, across different values of $k$.
This is encouraging, since it suggests that \minCAnn produces results
that are not significantly far from the optimal.
Interestingly, no optimal solution exists for $k>5$ for the regional
event: this credibility threshold  experiences a ``saturation'', since
there  does not exist a set of reporters who can collectively satisfy
that threshold.
Other events saturate at different values of $k$.
Finally, while this is not apparent from these graphs, the minimum
cost solution is approximately linear in $k$ for \minCAnn and
\minCDP.

More interestingly, from Figure~\ref{fig:new5b}, it is clear that
the runtime of \minCAnn is \emph{2-3 orders of magnitude} lower than
that of \minCDP with the discretization setting $|\script{W}| =
1000W$.
This difference is not just a matter of degree, but may make the
difference between a useful application and one that is not useful:
\minCDP can take several \emph{tens of seconds} to complete while
\minCAnn takes at most a few hundred milliseconds, which might make
the difference between victory and defeat in a balloon hunt, or life
and death in a disaster response swarm!
The explanation for the performance difference is the lower asymptotic
complexity of \minCAnn.
A subtle finding is that the running time of both \minCAnn and \minCDP
decreases, sometimes dramatically in the case of \minCAnn, with
increasing $k$.
Intuitively, this is because there are fewer candidate sets of
reporters who can satisfy a higher credibility, resulting in a smaller
search space.

\begin{figure}[!t]
\centering {
    \subfigure[Optimality gap]{\includegraphics[width=1.72in]{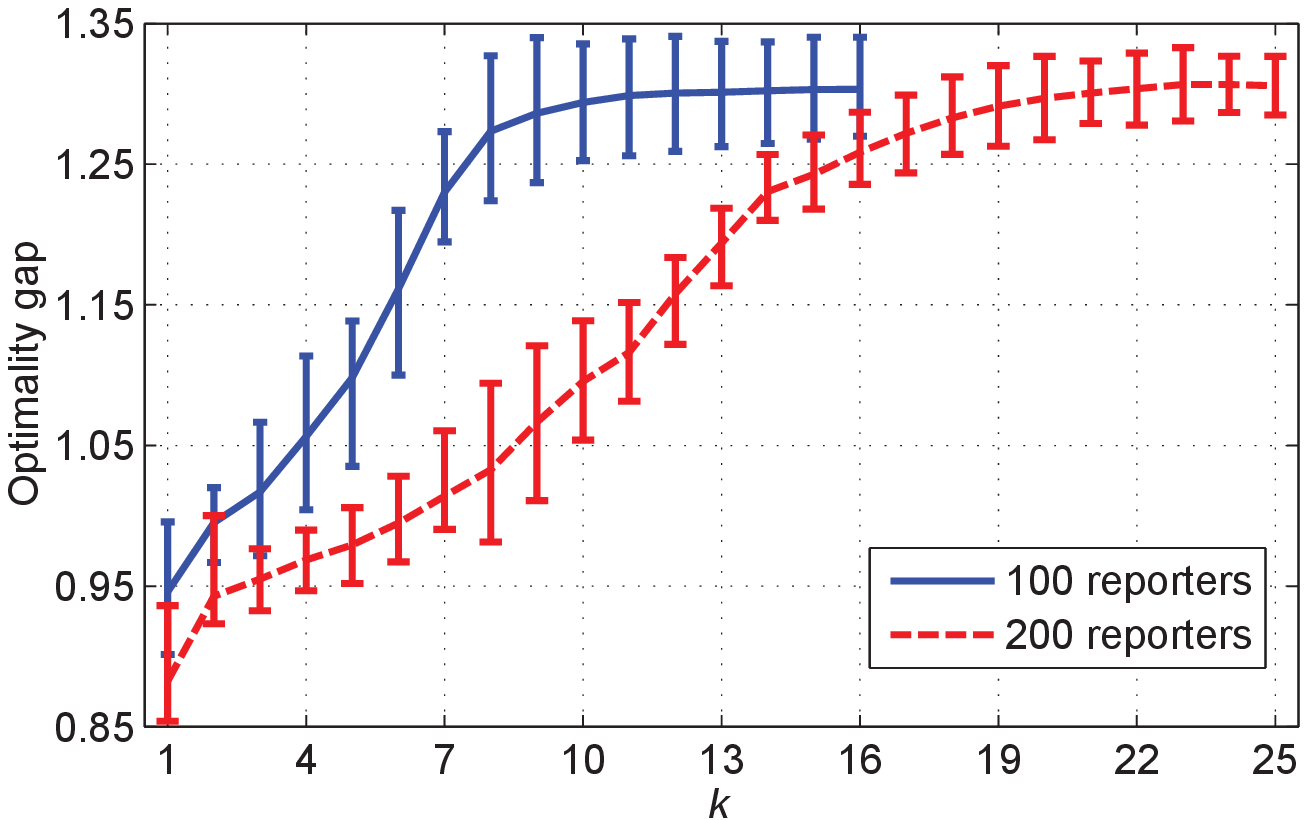}\label{fig:new6a}}
    \hfil
    \subfigure[Runtime (in microseconds)]{\includegraphics[width=1.68in]{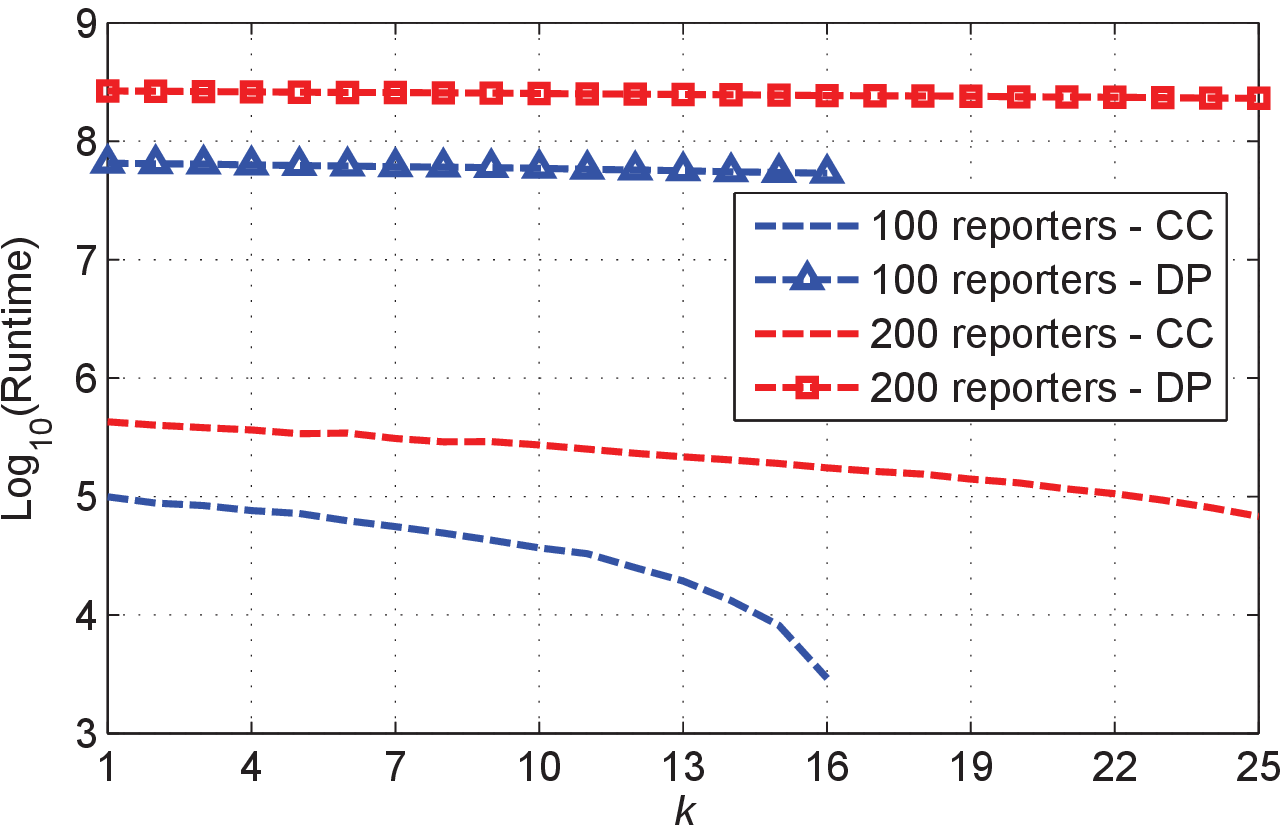}\label{fig:new6b}}
} \caption{Minimal cost in random topologies with increasing $k$
(error bars are very small thus ignored in (b))} \label{fig:new6}
\end{figure}

For random topologies, Figure~\ref{fig:new6} plots the optimality gap
and runtime, averaged over 50 simulations.
\minCAnn is, on average 20.5\% and 17.4\% for 100 and 200 reporters,
off the optimal for different values of $k$, but is still 2-3 orders of
magnitude more efficient than \minCDP.
The runtimes for both algorithms are slightly higher, given the
larger number of reporters.
Moreover, with 100 or 200 reporters, the optimality gap has the same
upper bound, about 35\% for large $k$.
This is also observed in other simulations for report numbers of 50,
150 and 300 (not shown).
We have left an analytical exploration of this upper bound to future work.
Finally, a comparison of these results with Figure~\ref{fig:new5a}
reveals an interesting result.
Although different types of reporter deployments can result in
different optimality gap curves (the curves for the three different
types of Google News in Figure~\ref{fig:new5a} are not the same), the
national event seems a qualitatively similar optimality gap curve as
the random topologies, suggesting that its deployment is similar to
that event.
Understanding this in greater depth is also left to future work.

\subsection{Extensions}
\label{sec:extensions}

Incorporating sources of noise into our algorithms is
straightforward, so we will mention this briefly.
Recall that the way we model a noise source increases a reporter's
effective distance.
Since our optimal algorithms, like \minCDP or \minCMCF, are agnostic
to the structure of the credibility function, they are unaffected by
noise.
For an algorithm like \minCAnn, which does take structure into
account, recall that noise sources increase a reporter's effective
distance.
Since reporters can quantify ambient noise, they can each use the
effective distance to calculate the report format to use.

Finally, our algorithms can, in general, deal with monotonically
increasing corroboration functions where the total credibility of a
collection of reporters may be a non-linear function of the individual
credibilities.
If $I(.)$ were to represent a monotonically increasing credibility
function, we only need use $I(c)$ to replace $c$ in our dynamic
programming formulation.
For example, (\ref{eq:dy}) would become
\begin{equation}
{A(l + 1,s)} = \max \left\{ {{A(l,s)},{c_l} + {A(l,I(s - {e_l}))}}
\right\} \label{eq:dym} \nonumber
\end{equation}
Similar changes can be applied to other dynamic programming formulations.
%


\section{The Renewals Problem: Randomly Arriving Events}
\label{sec:renewals}

In the previous section, we discussed a one-shot problem: that of
optimizing a single event.
We now consider a sequence of events that arrive at times $\{t_1, t_2,
t_3, \ldots\}$, where $t_k$ is a real number that represents the
arrival time of event $k$.
We assume that $t_k < t_{k+1}$ for all $k$.
In this setting, we consider a stochastic variant of \maxC, called
\maxCS: Instead of maximizing credibility for a single event subject
to a cost constraint, we maximize the average credibility-per-event
subject to an average cost constraint and a per-event credibility
minimum.
This couples the decisions needed for each event.
However, we first show that this time average problem can be solved by
a reduction to individual knapsack problems of the type described in
previous sections.
We then show that if the per-event credibility minimum is removed,
then \emph{decisions can be made in a decentralized fashion}.
Specifically, after the processing of every event, the swarm director
passes a weight to all reporters.
The reporters then make uncoordinated decisions when processing the
next event, without any intervention from the swarm director.
We derive a similar distributed version for stochastic $\minC$,
labelled $\minCS$.

We start by solving the general time average problem using Lyapunov
optimization \cite{now}, which can handle $\maxCS$, $\minCS$,
as well as variations with more general constraints.

\subsection{The  General Stochastic Problem}
\label{sec:renewals:general}

Let $\omega[k]$ represent a random vector of parameters associated
with each event $k$, such as the location of the event and the
corresponding costs and credibilities.
While $\omega[k]$ can include different parameters for different
types of problems, we shall soon use $\omega[k] \defequiv
[(c_{i,j}[k]), (e_{j}[k])]$, where $(c_{i,j}[k])$ is the matrix of
event-$k$ credibility values for reporters $i \in \{1, \ldots, N\}$
and formats $f_j \in \{f_1, \ldots, f_R\}$, and $(e_j[k])$ is a
vector of cost information.
We assume the process $\omega[k]$ is \emph{ergodic} with a well
defined steady-state distribution.
The simplest example is when $\omega[k]$ is independent and
identically distributed (i.i.d.)
over events $k \in \{1, 2, 3, \ldots\}$.

Let frame $k$ denote the period of time $[t_k, t_{k+1})$ which starts
with the arrival of event $k$ and ends just before the next event.
For every frame $k$, the director observes $\omega[k]$ and chooses a
\emph{control action} $\alpha[k]$ from a general set of feasible
actions $\script{A}_{\omega[k]}$ that possibly depend on $\omega[k]$.
The values $\omega[k]$ and $\alpha[k]$ together determine an $M+1$
dimensional vector $\bv{y}[k]$, representing \emph{network attributes}
for event $k$:
\[ \bv{y}[k] = (y_0[k], y_1[k], \ldots, y_M[k]) \]
Specifically, each $y_m[k]$ attribute is given by a general function
of $\alpha[k]$ and $\omega[k]$:
\[ y_m[k] = \hat{y}_m(\alpha[k], \omega[k]) \: \: \forall m \in \{0, 1, \ldots, M\} \]
The functions $\hat{y}_m(\alpha[k], \omega[k])$ are arbitrary
(possibly non-linear, non-convex, discontinuous), and are only
assumed to be bounded.

Define $\overline{y}_m$ as the time average expectation  of the
attribute $y_m[k]$, averaged over all frames (assuming temporarily
that the limit exists):
\begin{eqnarray*}
&\overline{y}_m \defequiv \lim_{k\rightarrow\infty}\frac{1}{K}\sum_{k=1}^K \expect{y_m[k]}&
\end{eqnarray*}

The general problem is to find an algorithm for choosing control
actions $\alpha[k]$ for each frame $k \in \{1, 2, 3, \ldots\}$ to
solve:
\begin{eqnarray}
\mbox{Minimize:} && \overline{y}_0 \\
\mbox{Subject to:} &1)& \overline{y}_m \leq 0 \: \: \forall m \in \{1, 2, \ldots, M\} \\
&2)& \alpha[k] \in \script{A}_{\omega[k]} \mbox{ $\forall$ frames $k \in \{1, 2, \ldots\}$}
\end{eqnarray}

The solution to the general problem is given in terms of a positive
parameter $V$, which affects a peformance tradeoff.
Specifically, for each of the $M$ time average inequality constraints
$\overline{y}_m \leq 0$, define a \emph{virtual queue} $Z_m[k]$ with
$Z_m[0] = 0$, and with frame-update equation:
\begin{equation} \label{eq:z-update}
 Z_m[k+1] = \max[Z_m[k] + y_m[k], 0]
 \end{equation}
Then every frame $k$, observe the value of $\omega[k]$ and perform
the following actions:
 \begin{itemize}
 \item Choose $\alpha[k] \in \script{A}_{\omega[k]}$ to minimize:
\begin{eqnarray*}
& V\hat{y}_0(\alpha[k], \omega[k]) + \sum_{m=1}^MZ_m[k]\hat{y}_m(\alpha[k], \omega[k]) &
\end{eqnarray*}

\item Update the virtual queues $Z_m[k]$ according to
  (\ref{eq:z-update}), using the values $y_m[k] = \hat{y}_m(\alpha[k],
  \omega[k])$ determined from the above minimization.
 \end{itemize}

Assuming the problem is feasible (so that it is possible to meet the
time average inequality constraints), this algorithm will \emph{also}
meet all of these constraints, and will achieve a time average value
$\overline{y}_0$ that is within $O(1/V)$ of the optimum.
Typically, the $V$ parameter also affects the average size of the
virtual queues (these can be shown to be $O(V)$, which directly
affects the \emph{convergence time} needed for the time averages to be
close to their limiting values).
The proofs of these claims follow the theory developed
in~\cite{now,neely-energy-it,neely-thesis}, with minor notational
adjustments needed to change the timeslot averages there to
frame-averages here. Specifically, the work in
\cite{now,neely-energy-it,neely-thesis} considers i.i.d.\ events
$\omega[k]$, but the same holds for more general ergodic events
\cite{stability}.

\subsection{Corroboration Pull as a Dynamic Optimization
  Problem}
\label{sec:renewals:corrob}

Here we formulate $\maxCS$.
Define $\omega[k] \defequiv [(c_{i,j}[k]), (e_{j}[k])]$
(representing costs and credibilities of event $k$), and define
$\alpha[k] \defequiv (x_{i,j}[k])$, where $x_{i,j}[k]$ is a binary
variable that is $1$ if reporter $i \in\{1, \ldots, N\}$ uses format
$f_j\in\{f_1, \ldots, f_R\}$ on frame $k$.
The goal is to maximize the average credibility-per-frame subject to
average cost constraints and to a minimum credibility level required
on each frame $k \in \{1, 2, \ldots\}$:
\begin{eqnarray}
\mbox{Maximize:} & \overline{c} \label{eq:dyn1} \\
\mbox{Subject to:} & \overline{e}  \leq e_{av} \label{eq:dyn2} \\
& \sum_{i=1}^N\sum_{j=1}^R x_{i,j}[k]c_{i,j}[k] \geq c_{min} \: \: \forall \mbox{frames $k$} \label{eq:dyn3} \\
& x_{i,j}[k] \in \{0, 1\} \: \: \forall i, j, \forall \mbox{frames $k$}   \label{eq:dyn4} \\
& \sum_{j=1}^R x_{i,j}[k] \leq 1 \: \: \forall j, \forall \mbox{frames $k$} \label{eq:dyn5}
\end{eqnarray}
where $e_{av}$ and $c_{min}$ are given non-negative constants, and
$\overline{c}$ and $\overline{e}$ are defined:
\begin{eqnarray*}
 \overline{c} &\defequiv& \lim_{K\rightarrow\infty} \frac{1}{K}\sum_{k=1}^K \sum_{i=1}^N\sum_{j=1}^R
 \expect{x_{i,j}[k]c_{i,j}[k]} \\
 \overline{e} &\defequiv& \lim_{K\rightarrow\infty} \frac{1}{K}\sum_{k=1}^K\sum_{i=1}^N\sum_{j=1}^R\expect{x_{i,j}[k]e_j[k]}
\end{eqnarray*}

This problem fits the general stochastic optimization framework of
the previous subsection by defining $y_0(t), y_1(t)$ by:
\begin{eqnarray*}
&y_0(t) = \hat{y}_0(\alpha(t), \omega(t)) \defequiv -\sum_{i=1}^N\sum_{j=1}^R x_{i,j}[k]c_{i,j}[k] \\
&y_1(t) = \hat{y}_1(\alpha(t), \omega(t)) \defequiv - e_{av} + \sum_{i=1}^N\sum_{j=1}^R x_{i,j}[k] e_j[k]
\end{eqnarray*} and by defining the set $\script{A}_{\omega[k]}$ as
the set of all $(x_{i,j}[k])$ matrices that satisfy the constraints
(\ref{eq:dyn3})-(\ref{eq:dyn5}).
The resulting stochastic algorithm thus defines a virtual queue
$Z_1(t)$ with update:
\begin{eqnarray*}
Z_1[k+1]= \max\left[Z_1[k] - e_{av} +  \sum_{i=1}^N\sum_{j=1}^R x_{i,j}[k] e_j[k] , 0\right]
\end{eqnarray*}
and then observing the $\omega[k]$ parameters every frame $k$ and
choosing $(x_{i,j}[k])$ to solve:
  \begin{eqnarray}
 \mbox{Minimize:} & \sum_{i=1}^N\sum_{j=1}^R x_{i,j}[k][Z_1[k]e_j[k] - Vc_{i,j}[k]] \label{eq:dyn-soln1} \\
 \mbox{Subject to:} & \sum_{i=1}^N\sum_{j=1}^R x_{i,j}[k]c_{i,j}[k] \geq c_{min} \label{eq:dyn-soln2} \\
& x_{i,j}[k] \in \{0, 1\} \: \: \forall i, j,  \sum_{j=1}^R x_{i,j}[k] \leq 1 \label{eq:dyn-soln3}
 \end{eqnarray} $\maxCS$ problem has the exact same
structure as the one-shot $\minC$ problem described in previous
sections, with the exception that the cost weights are changed from
$e_j[k]$ to $Z_1[k]e_j[k] - Vc_{i,j}[k]$, which can possibly be
negative.
However, the same knapsack technique can be used to solve it (possibly
by shifting the function to be minimized by a positive constant to
make the resulting terms non-negative).
Further, we note that the performance of the stochastic algorithm
degrades gracefully when \emph{approximate} implementations are used,
such as implementations that are off from the optimal knapsack problem
by a multiplicative constant~\cite{now}.
Thus, the simple \minCAnn heuristic can be used here.

A simple and exact distributed implementation arises if the
$c_{min}$ constraint (\ref{eq:dyn3}) is removed (i.e., if $c_{min}
\defequiv 0$).
In this case the frame $k$ decisions
(\ref{eq:dyn-soln1})-(\ref{eq:dyn-soln3}) are \emph{separable over
reporters} and reduce to having each reporter $i \in \{1, \ldots, N\}$
solve:
\begin{eqnarray*}
\mbox{Minimize:} & \sum_{j=1}^Rx_{i,j}[k][Z_1[k]e_j[k] - Vc_{i,j}[k]] \\
\mbox{Subject to:} & x_{i,j}[k] \in \{0, 1\} \: \forall j \: , \:  \sum_{j=1}^Rx_{i,j}[k] \leq 1
\end{eqnarray*} That is, each reporter $i$ chooses the single format
$f_j \in \{f_1, \ldots, f_R\}$ with the smallest value of
$Z_1[k]e_j[k] - Vc_{i,j}[k]$, breaking ties arbitrarily and choosing
to be idle (with $x_{i,j}[k]=0$ for all $j \in \{1, \ldots, R\}$) if
all of the weights $Z_1[k]e_j[k] - Vc_{i,j}[k]$ are positive.
The swarm director observes the outcomes of the decisions on frame $k$
and iterates the $Z_1[k]$ update, passing the weight $Z_1[k+1]$ to all
reporters before the next event occurs.

\subsection{$\minCS$}

$\minCS$ can be formulated as follows:
\begin{eqnarray*}
\mbox{Minimize:} & \overline{e} \\
\mbox{Subject to:} & \overline{c} \geq c_{av} \: \: ,
\mbox{and constraints (\ref{eq:dyn4}), (\ref{eq:dyn5})}
\end{eqnarray*}
We can thus define $y_0(t)$ and $y_1(t)$ as:
\begin{eqnarray*}
&y_0[k] \defequiv \sum_{i=1}^N\sum_{j=1}^Rx_{i,j}[k]e_j[k] \\
&y_1[k] \defequiv c_{av} - \sum_{i=1}^N\sum_{j=1}^Rx_{i,j}[k]c_{i,j}[k]
\end{eqnarray*}
from which a similar distributed solution can be obtained.


\section{Related Work}
\label{sec:related}

We are not aware of any prior work in the wireless networking
literature that has tackled information credibility assessment.

However, other fields have actively explored credibility, defined as
the believability of sources or information~\cite{r14,r20,r21}.
Credibility has been investigated in a number of fields including
information science, human communication, human-computer interaction
(HCI), marketing, psychology and so on \cite{r15}.
In general, research has focused on two threads: the factors that
affect credibility, and the dynamics of information credibility.

The seminal work of Hovland et al.~\cite{r22} may be the earliest
attempt on exploring credibility, which discusses how the various
characteristics of a source can affect a recipient's acceptance of
a message, in the context of human communication.
Rieh, Hilligoss and other explore important dimensions of
credibility in the context of social interactions~\cite{r14, r15,
r19}, such as trustworthiness, expertise and information validity.
McKnight and Kacmar~\cite{r14} study a unifying framework of
credibility assessment in which three distinct levels of credibility
are discussed: construct, heuristics, and interaction.
Their work is in the context of assessing the credibility of websites
as sources of information.

Wright and Laskey~\cite{r16} discuss how to tackle fusion of
credible information.
They present a weighting based, probabilistic model to compute
uncertain information credibility from diverse sources.
Several techniques are combined with this model, like prior
information, evidence when available and opportunities for learning
from data.

Sometimes, the terms credibility and trust are used synonymously.
However, they are distinct notions: while trust refers to beliefs and
behaviors associated with the acceptance of risk, credibility refers
to the believability of a source, and a believable source may or may
not result in associated trusting behaviors~\cite{r15}.

Finally, there is a body of work that has examined processes and
propagation of credible information.
Corroboration as a process of credibility assessment is discussed in
\cite{Cialdini04}.
Proximity, both geographic and social, and its role in credibility
assessment is discussed in~\cite{Dyads}: our role of geographic
distance as a measure of credibility is related to this discussion.
Saavedra et al.\cite{Saavedra10} explore the dynamics and the
emergence of synchronicity in decision-making when traders use
corroboration as a mechanism for trading decisions.


\section{Conclusions and Future Work}
\label{sec:concl-future-work}

In this paper, we have explored the design space of algorithms for a
new problem, optimizing pull corroboration in an emerging application
area, social swarming.
We have proposed optimal special-case algorithms, computationally
efficient approximations, and decentralized optimal stochastic
variants.
However, our work is merely an initial foray into a broad and
unexplored space, with several directions for future work: increasing
credibility and cost model realism, incorporating malice, allowing
peers to relay reports, and exploring other realistic, yet efficient
and near-optimal special-case solutions.



\end{document}